\documentclass{amsart}
\usepackage{amssymb}
\usepackage{amsfonts}

\newtheorem{theorem}{Theorem}
\theoremstyle{plain}

\newtheorem{definition}{Definition}
\newtheorem{example}{Example}

\newtheorem{lemma}{Lemma}

\newtheorem{remark}{Remark}

\numberwithin{equation}{section}
\input{tcilatex}

\begin{document}
\title[Existence of fixed points]{Existence of Fixed points for Condensing
Operators Under an Integral Condition}
\author{Vatan KARAKAYA}
\address{Department of Mathematical Engineering, Faculty of
Chemistry-Metallurgical, Yildiz Technical University, Istanbul, Turkey }
\email{vkkaya@yahoo.com}
\author{Nour El Houda BOUZARA}
\address{Department of Mathematics, Faculty of Science and Letters, Yildiz
Technical University, Istanbul, Turkey }
\email{bzr.nour@gmail.com}
\author{Kadri DOGAN}
\address{Department of Mathematical Engineering, Faculty of
Chemistry-Metallurgical, Yildiz Technical University, Istanbul, Turkey }
\email{dogankadri@hotmail.com}
\author{Yunus ATALAN}
\address{Department of Mathematics, Faculty of Science and Letters, Yildiz
Technical University, Istanbul, Turkey }
\email{yunus\_atalan@hotmail.com}

\begin{abstract}
Our aim in this paper is to present results of existence of fixed points for
continuous operators in Banach spaces using measure of noncompactness under
an integral condition. This results are generalization of results given by
A. Aghajania and M. Aliaskaria in \cite{agha} which are generalization of
Darbo's fixed point theorem. As application we use these results to solve\
an integral equations in Banach spaces.
\end{abstract}

\maketitle

\section{INTRODUCTION AND PRELIMINARIES}

The Darbo's fixed point theorem which guarantees the existence of fixed
point for so called condensing mappings is very famous, since it generalizes
two important theorems : The Banach principle theorem and The classical
Schauder theorem. Over recent years, many authors presented works that give
generalization of this theorem, see Aghajani et al. in \cite{aghabanas}, 
\cite{agha}, \cite{mursaleen}, Samadi and Ghaemi in \cite{Samadi}, \cite%
{samadi} and others.

Throughout this paper, $X$ is assumed to be a Banach space and $BC\left( 
\mathbb{R}^{+}\right) $ is the space of all real functions defined, bounded
and continuous on $\mathbb{R}^{+}$. The family of bounded subset, closure
and closed convex hull of $X$ are denoted by $\mathcal{B}_{X}$, $\overline{X}
$ and $ConvX$, respectively.

\begin{definition}
\cite{banas2} Let $X$ be a Banach space and $\mathcal{B}_{X}$ the family of
bounded subset of $X.$ A map%
\begin{equation*}
\mu :\mathcal{B}_{X}\rightarrow \left[ 0,\infty \right)
\end{equation*}%
is called measure of noncompactness defined on $X$ if it satisfies the
following:

\begin{enumerate}
\item $\mu \left( A\right) =0\Leftrightarrow A$ is a precompact set.

\item $A\subset B\Rightarrow \mu \left( A\right) \leqslant \mu \left(
B\right) .$

\item $\mu \left( A\right) =\mu \left( \overline{A}\right) ,$ $\forall A\in 
\mathcal{B}_{X}.$

\item $\mu \left( ConvA\right) =\mu \left( A\right) .$

\item $\mu \left( \lambda A+\left( 1-\lambda \right) B\right) \leqslant
\lambda \mu \left( A\right) +\left( 1-\lambda \right) \mu \left( B\right) ,$
for $\lambda \in \left[ 0,1\right] .$

\item Let $\left( A_{n}\right) $ be a sequence of closed sets from $\mathcal{%
B}_{X}$\ such that $A_{n+1}\subseteq A_{n},$ $\left( n\geqslant 1\right) $
and $\lim\limits_{n\rightarrow \infty }\mu \left( A_{n}\right) =0$, then the
intersection set $A_{\infty }=\bigcap\limits_{n=1}^{\infty }A_{n}$ is
nonempty and $A_{\infty }$ is precompact.
\end{enumerate}
\end{definition}

\begin{definition}
A summable function is function for which the integral exists and is finite.
\end{definition}

\begin{definition}
Consider a function $f:\mathbb{R}\rightarrow \mathbb{R}$ and a point $%
x_{0}\in \mathbb{R}$. The function $f$ is said to be upper (resp. lower)
semi-continuous at the point $x_{0}$ if 
\begin{equation*}
f\left( x_{0}\right) \geqslant \lim \sup_{x\rightarrow x_{0}}f(x)\text{ \ \ }%
(\text{resp}.f(x_{0})\leq \lim \inf_{x\rightarrow x_{0}}f(x).
\end{equation*}
\end{definition}

\begin{lemma}
\cite{aghabanas}\label{lem} Let $\psi :\mathbb{R}^{+}\rightarrow \mathbb{R}%
^{+}$ be a nondecreasing and upper semi-continuous function. Then,

$\lim\limits_{n\rightarrow \infty }$ $\psi ^{n}\left( t\right) =0$ for each $%
t>0$ $\Leftrightarrow $ \ $\psi \left( t\right) <t$ for any $t>0$.
\end{lemma}

\begin{theorem}[Banach contraction theorem]
\cite{agarwal} Let $X$ be a Banach space and $T:X\rightarrow X$ be a
contraction mapping on $X$, i.e. there is a nonnegative real number $k<1$
such that 
\begin{equation*}
\left\Vert Tx-Ty\right\Vert \leqslant k\left\Vert x-y\right\Vert \text{ \ \
\ }\forall x,y\in X.
\end{equation*}%
Then the map $T$ admits one and only one fixed point $x^{\ast }$ in $X$.
\end{theorem}

\begin{theorem}[Schauder theorem]
\cite{agarwal} Let $A$ be a nonempty, convex, compact subset of a Banach
space $X$ and suppose $T:A\rightarrow A$ is continuous. Then $T$ has a fixed
point.
\end{theorem}

The famous Darbo's theorem is as following

\begin{theorem}[Darbo's theorem]
\cite{lec.note} Let $C$ be a nonempty closed, bounded and convex subset of $%
X $. If $T:C\rightarrow C$ is a continuous mapping 
\begin{equation*}
\mu \left( TA\right) \leqslant k\mu \left( A\right) ,\text{ \ }k\in \left[
0,1\right) \text{,}
\end{equation*}%
then T has a fixed point.
\end{theorem}

\begin{theorem}
\cite{agha} Let $X$ be a Banach space and $A$ be a nonempty, closed, bounded
and convex subset of a Banach space $X$ and let $T:A\rightarrow A$ be a
continuous \ operator which satisfies the following inequality%
\begin{equation*}
\int_{0}^{\mu \left( TX\right) }\varphi \left( \gamma \right) d\gamma
\leqslant \Psi \left( \int_{0}^{\mu \left( X\right) }\varphi \left( \gamma
\right) d\gamma \right) ,
\end{equation*}%
where $\mu $ is a measure of noncompactness, $\Psi :\mathbb{R}%
^{+}\rightarrow \mathbb{R}^{+}$ a nondecreasing function such that $%
\lim\limits_{n\rightarrow \infty }\Psi ^{n}\left( t\right) =0,$ $\forall
t\geqslant 0$ and $\varphi :\left[ 0,+\infty \right[ \rightarrow \left[
0,+\infty \right[ $ is integral mapping which is summable on each compact
subset of $\left[ 0,+\infty \right[ $ and for each $\epsilon >0,$ $%
\int_{0}^{\epsilon }\varphi \left( \gamma \right) d\gamma >0.$

Then $T$ has at least one fixed point in $X$.
\end{theorem}

\section{\protect\bigskip MAIN RESULTS}

\begin{theorem}
\label{mainth}Let $X$ be a Banach space and $A$ be a nonempty, closed,
bounded and convex subset of a Banach space $X$ and let $T:A\rightarrow A$
be a continuous \ operator which satisfies the following inequality%
\begin{equation}
\Phi \left( \int_{0}^{\mu \left( TX\right) }\varphi \left( \gamma \right)
d\gamma \right) \leqslant \Psi \left( \int_{0}^{\mu \left( X\right) }\varphi
\left( \gamma \right) d\gamma \right) ,  \label{ineq}
\end{equation}%
where $\mu $ is a measure of noncompactness and

\begin{enumerate}
\item[$\left( i\right) $] $\Psi :\mathbb{R}^{+}\rightarrow \mathbb{R}^{+}$
is a nondecreasing and concave function such that $\lim\limits_{n\rightarrow
\infty }\Psi ^{n}\left( t\right) =0,$ $\forall t\geqslant 0.$

\item[$\left( ii\right) $] $\Phi :\mathbb{R}^{+}\rightarrow \mathbb{R}^{+}$
is a nondecreasing subadditive function such that $\Phi \left( t\right)
\geqslant t$ and$\ \lim\limits_{n\rightarrow \infty }\Phi \left(
x_{n}\right) =0\Leftrightarrow \lim\limits_{n\rightarrow \infty }x_{n}=0$.

\item[$\left( iii\right) $] $\varphi :\left[ 0,+\infty \right[ \rightarrow %
\left[ 0,+\infty \right[ $ is an integral mapping which is summable on each
compact subset of $\left[ 0,+\infty \right[ $ and for each $\epsilon >0,$ $%
\int_{0}^{\epsilon }\varphi \left( \omega \right) d\omega >0.$
\end{enumerate}

Then $T$ has at least one fixed point in $X$.
\end{theorem}

\begin{proof}
Consider $\left( A_{n}\right) _{n=0}^{\infty }$, a closed and convex
sequence of subset of $X$ such that $A_{n+1}=Conv\left( TA_{n}\right) $. We
notice that $A_{1}=Conv\left( TA_{0}\right) \subseteq A_{0}$ and $%
A_{2}=Conv\left( TA_{1}\right) \subseteq A_{1}$. By induction, we get 
\begin{equation*}
...A_{n+1}\subseteq A_{n}\subseteq ...\subseteq A_{0}.
\end{equation*}%
In further, we have%
\begin{eqnarray*}
\int_{0}^{\mu \left( A_{n+1}\right) }\varphi \left( \gamma \right) d\gamma
&=&\int_{0}^{\mu \left( Conv\left( TA_{n}\right) \right) }\varphi \left(
\gamma \right) d\gamma \\
&=&\int_{0}^{\mu \left( TA_{n}\right) }\varphi \left( \gamma \right) d\gamma
.
\end{eqnarray*}%
Using $\left( \text{\ref{ineq}}\right) ,$ we get 
\begin{equation}
\Phi \left( \int_{0}^{\mu \left( A_{n+1}\right) }\varphi \left( \gamma
\right) d\gamma \right) =\Phi \left( \int_{0}^{\mu \left( TA_{n}\right)
}\varphi \left( \gamma \right) d\gamma \right) \leqslant \Psi \left(
\int_{0}^{\mu \left( A_{n}\right) }\varphi \left( \gamma \right) d\gamma
\right) .  \label{pr1}
\end{equation}%
Moreover, using $\Phi \left( t\right) \geqslant t$ we get 
\begin{equation*}
\int_{0}^{\mu \left( A_{n}\right) }\varphi \left( \gamma \right) d\gamma
\leqslant \Phi \left( \int_{0}^{\mu \left( A_{n}\right) }\varphi \left(
\gamma \right) d\gamma \right) ,
\end{equation*}%
and%
\begin{eqnarray*}
\int_{0}^{\mu \left( A_{n}\right) }\varphi \left( \gamma \right) d\gamma
&=&\int_{0}^{\mu \left( Conv\left( TA_{n-1}\right) \right) }\varphi \left(
\gamma \right) d\gamma \\
&=&\int_{0}^{\mu \left( TA_{n-1}\right) }\varphi \left( \gamma \right)
d\gamma .
\end{eqnarray*}%
Then,%
\begin{eqnarray*}
\Psi \left( \int_{0}^{\mu \left( A_{n}\right) }\varphi \left( \gamma \right)
d\gamma \right) &=&\Psi \left( \int_{0}^{\mu \left( TA_{n-1}\right) }\varphi
\left( \gamma \right) d\gamma \right) \\
&\leqslant &\Psi \left( \Phi \left( \int_{0}^{\mu \left( TA_{n-1}\right)
}\varphi \left( \gamma \right) d\gamma \right) \right) \\
&\leqslant &\Psi ^{2}\left( \int_{0}^{\mu \left( A_{n-1}\right) }\varphi
\left( \gamma \right) d\gamma \right) .
\end{eqnarray*}%
Repeating this process $n$ times we get%
\begin{equation*}
\Phi \left( \int_{0}^{\mu \left( A_{n+1}\right) }\varphi \left( \gamma
\right) d\gamma \right) \leqslant \Psi ^{n+1}\left( \int_{0}^{\mu \left(
A_{0}\right) }\varphi \left( \gamma \right) d\gamma \right) .
\end{equation*}%
In view of condition $\left( i\right) ,$ we get $\lim\limits_{n\rightarrow
\infty }\Psi ^{n+1}\left( \int_{0}^{\mu \left( A_{0}\right) }\varphi \left(
\gamma \right) d\gamma \right) =0$. Then,%
\begin{equation*}
\lim\limits_{n\rightarrow \infty }\Phi \left( \int_{0}^{\mu \left(
A_{n+1}\right) }\varphi \left( \gamma \right) d\gamma \right) =0,
\end{equation*}%
using $\left( ii\right) $ we obtain%
\begin{equation*}
\lim\limits_{n\rightarrow \infty }\int_{0}^{\mu \left( A_{n+1}\right)
}\varphi \left( \gamma \right) d\gamma =0.
\end{equation*}%
It follows that%
\begin{equation*}
\lim\limits_{n\rightarrow \infty }\mu \left( A_{n+1}\right) =0.
\end{equation*}%
Consequently, $A_{\infty }$ is compact and then $T$ has at least one fixed
point.
\end{proof}

\begin{remark}

\begin{enumerate}
\item[$\left( i\right) $] $\Phi \left( x\right) =x$, then $\Phi :\mathbb{R}%
^{+}\rightarrow \mathbb{R}^{+}$ is nondecreasing additive $\left( \text{%
hence subadditive}\right) $ function and$\ \lim\limits_{n\rightarrow \infty
}\Phi \left( x_{n}\right) =0\Leftrightarrow \lim\limits_{n\rightarrow \infty
}x_{n}=0.$ Thus, $\Phi \left( x\right) =x$ satisfies condition $\left(
ii\right) $ of theorem $\left( \text{\ref{mainth}}\right) $ and inequality $%
\left( \text{\ref{ineq}}\right) $ will be%
\begin{equation*}
\int_{0}^{\mu \left( TX\right) }\varphi \left( \gamma \right) d\gamma
\leqslant \Psi \left( \int_{0}^{\mu \left( X\right) }\varphi \left( \gamma
\right) d\gamma \right) ,
\end{equation*}%
which is the condition given by Aghajani and Aliaskari in \cite{agha}.

\item[$\left( ii\right) $] $\Phi \left( x\right) =x,$ and $\Psi \left(
x\right) =kx,$ where $k\in \left[ 0,1\right[ .$ Since $\lim\limits_{n%
\rightarrow \infty }\Psi ^{n}\left( x\right) =\lim\limits_{n\rightarrow
\infty }k^{n}x=0.$\ Then, $\Psi $ and $\Phi $\ satisfy conditions $\left(
i\right) $ and $\left( ii\right) $ of theorem $\left( \text{\ref{mainth}}%
\right) .$ Then, inequality $\left( \text{\ref{ineq}}\right) $ become 
\begin{equation*}
\int_{0}^{\mu \left( TX\right) }\varphi \left( \gamma \right) d\gamma
\leqslant k\int_{0}^{\mu \left( X\right) }\varphi \left( \gamma \right)
d\gamma ,
\end{equation*}%
which is a generalization of the result given by Branciari in \cite{branc}.

\item[$\left( iii\right) $] $\Phi \left( x\right) =x,$ $\Psi \left( x\right)
=kx$ and $\varphi \left( x\right) =1.$ These functions satisfy conditions $%
\left( i\right) -\left( iii\right) $ of theorem $\left( \text{\ref{mainth}}%
\right) $ and instead of inequality $\left( \text{\ref{ineq}}\right) $ we
obtain the following inequality%
\begin{equation*}
\mu \left( TX\right) \leqslant k\mu \left( X\right) ,
\end{equation*}%
which is the condition given by Darbo in his famous fixed point theorem $%
\left( \text{see \cite{lec.note}}\right) $.
\end{enumerate}
\end{remark}

\section{APPLICATIONS}

In this section we use Theorem \ref{mainth} to study the resolvability of
the following integral equation in the Banach space $BC\left( \mathbb{R}%
^{+}\right) $ under more general hypothesis.%
\begin{equation}
x\left( t\right) =f\left( t,\int_{0}^{t}g\left( t,s,x\left( s\right) \right)
ds,x\left( t\right) \right) ,\text{ }t\in \mathbb{R}^{+}.  \label{eq}
\end{equation}%
In what follow we formulate the assumptions under which equation $\left( 
\text{\ref{eq}}\right) $ will be studied:

\begin{enumerate}
\item[$\left( i\right) $] The function $f:\mathbb{R}^{+}\times \mathbb{R}%
^{+}\times \mathbb{R}^{+}\rightarrow \mathbb{R}^{+}$ is continuous and $%
|f(t,x,0)|\in BC(\mathbb{R}^{+})$ for $t\in \mathbb{R}^{+}$, and $x\in 
\mathbb{R}$.

\item[$\left( ii\right) $] There exist a nondecreasing, concave and upper
semi-continuous function $\psi :\mathbb{R}^{+}\rightarrow \mathbb{R}^{+}$and
a nondecreasing, upper semi-continuous and sub-additive function $\phi :$ $%
\mathbb{R}^{+}\rightarrow \mathbb{R}^{+}$ such that $\Phi \left( t\right)
\geqslant t$ and $\lim\limits_{n\rightarrow \infty }\phi \left( x_{n}\right)
=0\Leftrightarrow \lim\limits_{n\rightarrow \infty }x_{n}=0,$ for which the
function $f:\mathbb{R}^{+}\times \mathbb{R}^{+}\times \mathbb{R}%
^{+}\rightarrow \mathbb{R}^{+}$ satisfies the conditions%
\begin{align*}
\Phi \left( \int^{\left\vert f\left( t,x,y_{1}\right) -f\left(
t,x,y_{2}\right) \right\vert }\varphi \left( \gamma \right) d\gamma \right)
& \leqslant \Psi \left( \int^{\left\vert y_{1}-y_{2}\right\vert }\varphi
\left( \gamma \right) d\gamma \right) , \\
\Phi \left( \int^{\left\vert f\left( t,x_{1},y\right) -f\left(
t,x_{1},y\right) \right\vert }\varphi \left( \gamma \right) d\gamma \right)
& \leqslant \Psi \left( \int^{\left\vert x_{1}-x_{2}\right\vert }\varphi
\left( \gamma \right) d\gamma \right) ,
\end{align*}%
where $\varphi :[0,\infty \lbrack \rightarrow \lbrack 0,\infty \lbrack $ is
summable on every compact subset of $[0,\infty \lbrack $ and for every $%
\epsilon >0$, $\int_{0}^{\epsilon }\varphi \left( \omega \right) d\omega >0.$

\item[$\left( iii\right) $] The function $g:\mathbb{R}^{+}\times \mathbb{R}%
^{+}\times \mathbb{R}\rightarrow \mathbb{R}$ is continuous and there exist
continuous functions $a,b:\mathbb{R}^{+}\rightarrow \mathbb{R}^{+}$ such
that $\lim\limits_{t\rightarrow \infty }a\left( t\right) =0$, $b\in L_{1}(%
\mathbb{R}^{+})$ and $|g(t,s,x)|\leqslant a(t)b(s)$ for $t,s\in \mathbb{R}%
^{+}$ such that $s\leqslant t$ and each $x\in \mathbb{R}.$

\item[$\left( iv\right) $] There exists at least one positive constant $%
r_{0} $ such that the following inequality holds,%
\begin{equation*}
\int_{0}^{r}\varphi \left( \gamma \right) d\gamma \leqslant
\int_{0}^{r}\varphi \left( \gamma \right) d\gamma +M_{0}+M_{1},
\end{equation*}%
where $M_{0}=\sup \left\{ \int_{0}^{\left\vert a\left( t\right) \right\vert
\int_{0}^{t}\left\vert b\left( s\right) \right\vert ds}\varphi \left( \gamma
\right) d\gamma \right\} $ and $M_{1}=\Phi \left( \int_{0}^{\sup \left\vert
f\left( t,0,0\right) \right\vert }\varphi \left( \gamma \right) d\gamma
\right) $.
\end{enumerate}

\begin{theorem}
Under the hypothesis $\left( i\right) -\left( iv\right) $ the integral
equation $\left( \text{\ref{eq}}\right) $ has at least one solution in the
space $BC\left( \mathbb{R}^{+}\right) $.
\end{theorem}

\begin{proof}
Study the solvability of equation $\left( \text{\ref{eq}}\right) $ is
equivalent to study the existence of fixed points of the following operator%
\begin{equation*}
Tx\left( t\right) =f\left( t,\int_{0}^{t}g\left( t,s,x\left( s\right)
\right) ds,x\left( t\right) \right) ,\text{ }t\in \mathbb{R}^{+}.
\end{equation*}%
For that we need to verify that under assumptions $\left( i\right) -\left(
iv\right) $ the operator $T$ satisfies the conditions of Theorem \ref{mainth}%
.

First, let recall the following notions, the measure of noncompactness for a
positive fixed $t$ on $\mathcal{B}_{BC\left( \mathbb{R}^{+}\right) }$ is
given by the following%
\begin{equation*}
\mu \left( X\right) =\omega _{0}\left( X\right) +\lim \sup_{t\rightarrow
\infty }diamX\left( t\right) ,
\end{equation*}%
where,%
\begin{equation*}
diamX\left( t\right) =\sup \left\{ \left\vert x\left( t\right) -y\left(
t\right) \right\vert :x,y\in X\right\} ,\text{ }X\left( t\right) =\left\{
x\left( t\right) :x\in X\right\} ,
\end{equation*}%
and%
\begin{equation*}
\omega _{0}\left( X\right) =\lim_{L\rightarrow \infty }\omega _{0}^{L}\left(
X\right) .
\end{equation*}%
\begin{equation*}
\omega _{0}^{L}\left( X\right) =\lim_{\epsilon \rightarrow 0}\omega
^{L}\left( X,\epsilon \right) ,
\end{equation*}%
\begin{equation*}
\omega ^{L}\left( X,\epsilon \right) =\sup \left\{ \omega ^{L}\left(
x,\epsilon \right) :x\in X\right\} ,
\end{equation*}%
\begin{equation*}
\omega ^{L}\left( x,\epsilon \right) =\sup \left\{ \left\vert x\left(
t\right) -x\left( s\right) \right\vert :t,s\in \left[ 0,L\right] ,\text{ }%
\left\vert t-s\right\vert \leqslant \epsilon \right\} \text{, for }L>0.
\end{equation*}%
To show that $T$ is self-mappings, that is, $T$ map a ball $B_{r_{0}}$ into
itself, let%
\begin{multline*}
\int_{0}^{\left\vert Tx\left( t\right) \right\vert }\varphi \left( \gamma
\right) d\gamma \leqslant \Phi \left( \int_{0}^{\left\vert Tx\left( t\right)
\right\vert }\varphi \left( \gamma \right) d\gamma \right) \\
=\Phi \left( \int_{0}^{\left\vert f\left( t,\int_{0}^{t}g\left( t,s,x\left(
s\right) \right) ds,x\left( t\right) \right) \right\vert }\varphi \left(
\gamma \right) d\gamma \right) \\
\leqslant \Phi \left( \int_{0}^{\left\vert f\left( t,\int_{0}^{t}g\left(
t,s,x\left( s\right) \right) ds,x\left( t\right) \right) -f\left(
t,\int_{0}^{t}g\left( t,s,x\left( s\right) \right) ds,0\right) \right\vert
+\left\vert f\left( t,\int_{0}^{t}g\left( t,s,x\left( s\right) \right)
ds,0\right) -f\left( t,0,0\right) \right\vert +\left\vert f\left(
t,0,0\right) \right\vert }\varphi \left( \gamma \right) d\gamma \right) \\
\leqslant \Phi \left( \int_{0}^{\left\vert f\left( t,\int_{0}^{t}g\left(
t,s,x\left( s\right) \right) ds,x\left( t\right) \right) -f\left(
t,\int_{0}^{t}g\left( t,s,x\left( s\right) \right) ds,0\right) \right\vert
}\varphi \left( \gamma \right) d\gamma \right) \\
+\Phi \left( \int_{0}^{\left\vert f\left( t,\int_{0}^{t}g\left( t,s,x\left(
s\right) \right) ds,0\right) -f\left( t,0,0\right) \right\vert }\varphi
\left( \gamma \right) d\gamma \right) +\Phi \left( \int_{0}^{\left\vert
f\left( t,0,0\right) \right\vert }\varphi \left( \gamma \right) d\gamma
\right) .
\end{multline*}%
Using $\left( ii\right) ,$ we get%
\begin{equation*}
\int_{0}^{\left\vert Tx\left( t\right) \right\vert }\varphi \left( \omega
\right) d\omega \leqslant \Psi \left( \int_{0}^{\left\vert x\left( t\right)
\right\vert }\varphi \left( \gamma \right) d\gamma \right) +\Psi \left(
\int_{0}^{\left\vert \int_{0}^{t}g\left( t,s,x\left( s\right) \right)
ds\right\vert }\varphi \left( \gamma \right) d\gamma \right) +\Phi \left(
\int_{0}^{\sup\limits_{t}\left\vert f\left( t,0,0\right) \right\vert
}\varphi \left( \gamma \right) d\gamma \right)
\end{equation*}%
In view of condition $\left( iii\right) $ and Lemma \ref{lem}, we obtain%
\begin{eqnarray*}
\int_{0}^{\left\vert Tx\left( t\right) \right\vert }\varphi \left( \gamma
\right) d\gamma &\leqslant &\int_{0}^{\left\vert x\left( t\right)
\right\vert }\varphi \left( \gamma \right) d\gamma +\int_{0}^{\left\vert
a\left( t\right) \right\vert \int_{0}^{t}\left\vert b\left( s\right)
\right\vert ds}\varphi \left( \gamma \right) d\gamma +M_{1} \\
&\leqslant &\int_{0}^{\left\vert x\left( t\right) \right\vert }\varphi
\left( \gamma \right) d\gamma +M_{0}+M_{1}.
\end{eqnarray*}

Finally, the assumption $\left( iv\right) $ guaranties the existence of a
constant $r_{0}$ such that $TB_{r_{0}}\subseteq B_{r_{0}}.$

Now, let show that $T$ satisfies Condition \ref{ineq} of Theorem \ref{mainth}%
.%
\begin{eqnarray*}
\Phi \left( \int_{0}^{\left\vert Tx\left( t\right) -Ty\left( t\right)
\right\vert }\varphi \left( \gamma \right) d\gamma \right) &\leqslant &\Phi
\left( \int_{0}^{\left\vert f\left( t,\int_{0}^{t}g\left( t,s,x\left(
s\right) \right) ds,x\left( t\right) \right) -f\left( t,\int_{0}^{t}g\left(
t,s,y\left( s\right) \right) ds,y\left( t\right) \right) \right\vert
}\varphi \left( \gamma \right) d\gamma \right) \\
&\leqslant &\Phi \left( \int_{0}^{\substack{ \left\vert f\left(
t,\int_{0}^{t}g\left( t,s,x\left( s\right) \right) ds,x\left( t\right)
\right) -f\left( t,\int_{0}^{t}g\left( t,s,y\left( s\right) \right)
ds,x\left( t\right) \right) \right\vert }}\varphi \left( \gamma \right)
d\gamma \right) \\
&&+\Phi \left( \int_{0}^{\left\vert f\left( t,\int_{0}^{t}g\left(
t,s,y\left( s\right) \right) ds,x\left( t\right) \right) -f\left(
t,\int_{0}^{t}g\left( t,s,y\left( s\right) \right) ds,y\left( t\right)
\right) \right\vert }\varphi \left( \gamma \right) d\gamma \right) .
\end{eqnarray*}%
Using assumption $\left( ii\right) ,$ we get%
\begin{equation}
\Phi \left( \int_{0}^{\left\vert Tx\left( t\right) -Ty\left( t\right)
\right\vert }\varphi \left( \gamma \right) d\gamma \right) \leqslant \Psi
\left( \int_{0}^{\substack{ \left\vert x\left( t\right) -y\left( t\right)
\right\vert }}\varphi \left( \gamma \right) d\gamma \right) +\Psi \left(
\int_{0}^{\substack{ \int_{0}^{t}\left\vert g\left( t,s,x\left( s\right)
\right) -g\left( t,s,y\left( s\right) \right) \right\vert ds}}\varphi \left(
\gamma \right) d\gamma \right)  \label{condit}
\end{equation}%
Moreover,%
\begin{eqnarray*}
\Psi \left( \int_{0}^{\substack{ \int_{0}^{t}\left\vert g\left( t,s,x\left(
s\right) \right) -g\left( t,s,y\left( s\right) \right) \right\vert ds}}%
\varphi \left( \gamma \right) d\gamma \right) &\leqslant &\Psi \left(
\int_{0}^{\int_{0}^{t}\left\vert g\left( t,s,x\left( s\right) \right)
-g\left( t,s,y\left( s\right) \right) \right\vert ds}\varphi \left( \gamma
\right) d\gamma \right) \\
&<&\int_{0}^{\int_{0}^{t}\left\vert g\left( t,s,x\left( s\right) \right)
-g\left( t,s,y\left( s\right) \right) \right\vert ds}\varphi \left( \gamma
\right) d\gamma \\
&\leqslant &\int_{0}^{\substack{ \int_{0}^{t}\left\vert g\left( t,s,x\left(
s\right) \right) \right\vert ds+\int_{0}^{t}\left\vert g\left( t,s,y\left(
s\right) \right) \right\vert ds}}\varphi \left( \gamma \right) d\gamma \\
&\leqslant &\int_{0}^{\substack{ \int_{0}^{t}\left\vert g\left( t,s,x\left(
s\right) \right) \right\vert ds+\int_{0}^{t}\left\vert g\left( t,s,y\left(
s\right) \right) \right\vert ds}}\varphi \left( \gamma \right) d\gamma \\
&\leqslant &\int_{0}^{\substack{ 2\left\vert a\left( t\right) \right\vert
\int_{0}^{t}\left\vert b\left( s\right) \right\vert ds}}\varphi \left(
\gamma \right) d\gamma .
\end{eqnarray*}%
Since $\lim\limits_{t\rightarrow \infty }a\left( t\right) =0$ and $b\in
L_{1}(\mathbb{R}^{+})$, then for $t\geqslant L$ $\left( \text{where }L\text{
is positve constant}\right) $ we have 
\begin{equation*}
\int_{0}^{\substack{ 2\left\vert a\left( t\right) \right\vert
\int_{0}^{t}\left\vert b\left( s\right) \right\vert ds}}\varphi \left(
\gamma \right) d\gamma \leqslant \epsilon ,
\end{equation*}%
where $\epsilon $ is an arbitrary positive number.

Consequently, Inequality \ref{condit} will be%
\begin{equation*}
\Phi \left( \int_{0}^{\left\vert Tx\left( t\right) -Ty\left( t\right)
\right\vert }\varphi \left( \gamma \right) d\gamma \right) \leqslant \Psi
\left( \int_{0}^{\substack{ \left\vert x\left( t\right) -y\left( t\right)
\right\vert }}\varphi \left( \gamma \right) d\gamma \right) .
\end{equation*}%
For $t\leqslant L,$ we have%
\begin{equation*}
\omega _{1}^{L}\left( g,\epsilon \right) =\sup \left\{
\int_{0}^{t}\left\vert g\left( t,s,x\left( s\right) \right) -g\left(
t,s,y\left( s\right) \right) \right\vert ds:t,s\in \left[ 0,L\right] ,\text{ 
}x,y\in B_{r_{0}}\text{ and }\left\Vert x-y\right\Vert \leqslant \epsilon
\right\} \text{.}
\end{equation*}%
Thus for $t\in \left[ 0,L\right] ,$ we obtain%
\begin{equation*}
\Phi \left( \int_{0}^{\left\vert Tx\left( t\right) -Ty\left( t\right)
\right\vert }\varphi \left( \gamma \right) d\gamma \right) \leqslant \Psi
\left( \int_{0}^{\substack{ \left\vert x\left( t\right) -y\left( t\right)
\right\vert }}\varphi \left( \gamma \right) d\gamma \right) +\Psi \left(
\int_{0}^{\omega _{1}^{L}\left( g,\epsilon \right) }\varphi \left( \gamma
\right) d\gamma \right) .
\end{equation*}%
Since $g$ is continuous, it is uniformly continuous on $\left[ 0,L\right]
\times \left[ 0,L\right] \times \left[ -r_{0},r_{0}\right] $. Then,%
\begin{equation*}
\lim_{\epsilon \rightarrow 0}\omega _{1}^{L}\left( g,\epsilon \right) =0.
\end{equation*}%
Consequently, Inequality \ref{condit} will be%
\begin{equation*}
\Phi \left( \int_{0}^{\left\vert Tx\left( t\right) -Ty\left( t\right)
\right\vert }\varphi \left( \gamma \right) d\gamma \right) \leqslant \Psi
\left( \int_{0}^{\substack{ \left\vert x\left( t\right) -y\left( t\right)
\right\vert }}\varphi \left( \gamma \right) d\gamma \right) .
\end{equation*}%
Thus, for every $t\geqslant 0$ we have%
\begin{equation*}
\Phi \left( \int_{0}^{\left\vert Tx\left( t\right) -Ty\left( t\right)
\right\vert }\varphi \left( \gamma \right) d\gamma \right) \leqslant \Psi
\left( \int_{0}^{\substack{ \left\vert x\left( t\right) -y\left( t\right)
\right\vert }}\varphi \left( \gamma \right) d\gamma \right) ,
\end{equation*}%
hence,%
\begin{equation*}
\Phi \left( \int_{0}^{\lim \sup\limits_{t\rightarrow \infty }Diam\left(
TX\left( t\right) \right) }\varphi \left( \gamma \right) d\gamma \right)
\leqslant \Psi \left( \int_{0}^{\lim \sup\limits_{t\rightarrow \infty
}Diam\left( X\left( t\right) \right) }\varphi \left( \gamma \right) d\gamma
\right) .
\end{equation*}%
Now, let consider%
\begin{eqnarray*}
\Phi \left( \int_{0}^{\left\vert Tx\left( t\right) -Tx\left( l\right)
\right\vert }\varphi \left( \gamma \right) d\gamma \right) &=&\Phi \left(
\int_{0}^{\left\vert f\left( t,\int_{0}^{t}g\left( t,s,x\left( s\right)
\right) ds,x\left( t\right) \right) -f\left( l,\int_{0}^{l}g\left(
l,s,x\left( s\right) \right) ds,x\left( l\right) \right) \right\vert
}\varphi \left( \gamma \right) d\gamma \right) \\
&\leqslant &\Phi \left( \int_{0}^{\left\vert f\left( t,\int_{0}^{t}g\left(
t,s,x\left( s\right) \right) ds,x\left( t\right) \right) -f\left(
l,\int_{0}^{t}g\left( t,s,x\left( s\right) \right) ds,x\left( t\right)
\right) \right\vert }\varphi \left( \gamma \right) d\gamma \right) \\
&&+\Phi \left( \int_{0}^{\left\vert f\left( l,\int_{0}^{t}g\left(
t,s,x\left( s\right) \right) ds,x\left( t\right) \right) -f\left(
l,\int_{0}^{l}g\left( l,s,x\left( s\right) \right) ds,x\left( t\right)
\right) \right\vert }\varphi \left( \gamma \right) d\gamma \right) \\
&&+\Phi \left( \int_{0}^{\left\vert f\left( l,\int_{0}^{l}g\left(
l,s,x\left( s\right) \right) ds,x\left( t\right) \right) -f\left(
l,\int_{0}^{l}g\left( l,s,x\left( s\right) \right) ds,x\left( l\right)
\right) \right\vert }\varphi \left( \gamma \right) d\gamma \right) \\
&\leqslant &\Phi \left( \int_{0}^{\left\vert f\left( t,\int_{0}^{t}g\left(
t,s,x\left( s\right) \right) ds,x\left( t\right) \right) -f\left(
l,\int_{0}^{t}g\left( t,s,x\left( s\right) \right) ds,x\left( t\right)
\right) \right\vert }\varphi \left( \gamma \right) d\gamma \right) \\
&&+\Psi \left( \int_{0}^{\int_{0}^{t}\left\vert g\left( t,s,x\left( s\right)
\right) -g\left( l,s,x\left( s\right) \right) \right\vert ds}\varphi \left(
\gamma \right) d\gamma \right) +\Psi \left( \int_{0}^{\left\vert x\left(
t\right) -x\left( l\right) \right\vert }\varphi \left( \gamma \right)
d\gamma \right) .
\end{eqnarray*}%
Putting,%
\begin{eqnarray*}
\omega ^{L}\left( gx,\epsilon \right) &=&\sup \left\{ \left\vert g\left(
t,s,x\left( s\right) \right) -g\left( l,s,x\left( s\right) \right)
\right\vert :l,t,s\in \left[ 0,L\right] ,\text{ }x\in B_{r_{0}}\text{ and }%
\left\vert t-l\right\vert \leqslant \epsilon \right\} , \\
\omega ^{L}\left( fx,\epsilon \right) &=&\sup \left\{ \left\vert f\left(
t,x,y\right) -f\left( l,x,y\right) \right\vert :l,t\in \left[ 0,L\right] ,%
\text{ }x,y\in B_{r_{0}}\text{ and }\left\vert t-l\right\vert \leqslant
\epsilon \right\} , \\
\omega ^{L}\left( Tx,\epsilon \right) &=&\sup \left\{ \left\vert Tx\left(
t\right) -Tx\left( l\right) \right\vert :l,t\in \left[ 0,L\right] ,\text{ }%
x\in B_{r_{0}}\text{ and }\left\vert t-l\right\vert \leqslant \epsilon
\right\} ,
\end{eqnarray*}%
we get%
\begin{equation*}
\Phi \left( \int_{0}^{\omega ^{L}\left( Tx,\epsilon \right) }\varphi \left(
\gamma \right) d\gamma \right) \leqslant \Phi \left( \int_{0}^{\omega
^{L}\left( fx,\epsilon \right) }\varphi \left( \gamma \right) d\gamma
\right) +\Psi \left( \int_{0}^{\omega ^{L}\left( gx,\epsilon \right)
}\varphi \left( \gamma \right) d\gamma \right) +\Psi \left( \int_{0}^{\omega
^{L}\left( x,\epsilon \right) }\varphi \left( \gamma \right) d\gamma \right)
.
\end{equation*}%
We know that $g$ is uniformly continuous on $\left[ 0,L\right] \times \left[
0,L\right] \times \left[ -r_{0},r_{0}\right] $ and $f$ is uniformly
continuous on $\left[ 0,L\right] \times \left[ -r_{0},r_{0}\right] \times %
\left[ -r_{0},r_{0}\right] ,$ then we get%
\begin{equation*}
\Phi \left( \int_{0}^{\omega ^{L}\left( Tx,\epsilon \right) }\varphi \left(
\gamma \right) d\gamma \right) \leqslant \Psi \left( \int_{0}^{\omega
^{L}\left( x,\epsilon \right) }\varphi \left( \gamma \right) d\gamma \right)
.
\end{equation*}%
and then,%
\begin{equation*}
\Phi \left( \int_{0}^{\omega ^{L}\left( TX,\epsilon \right) }\varphi \left(
\gamma \right) d\gamma \right) \leqslant \Psi \left( \int_{0}^{\omega
^{L}\left( X,\epsilon \right) }\varphi \left( \gamma \right) d\gamma \right)
.
\end{equation*}%
By taking $\epsilon \rightarrow 0$, $L\rightarrow \infty $ and using the
fact that $\Phi $ and $\Psi $ are semicontinuous, we obtain%
\begin{equation*}
\Phi \left( \int_{0}^{\omega _{0}^{L}\left( TX\right) }\varphi \left( \gamma
\right) d\gamma \right) \leqslant \Psi \left( \int_{0}^{\omega
_{0}^{L}\left( X\right) }\varphi \left( \gamma \right) d\gamma \right) .
\end{equation*}%
Finally,%
\begin{eqnarray*}
\Phi \left( \int_{0}^{\mu \left( TX\right) }\varphi \left( \gamma \right)
d\gamma \right) &\leqslant &\Phi \left( \int_{0}^{\lim
\sup\limits_{t\rightarrow \infty }Diam\left( TX\left( t\right) \right)
}\varphi \left( \gamma \right) d\gamma \right) +\Phi \left( \int_{0}^{\omega
_{0}^{L}\left( TX\right) }\varphi \left( \gamma \right) d\gamma \right) \\
&\leqslant &\Psi \left( \int_{0}^{\lim \sup\limits_{t\rightarrow \infty
}Diam\left( X\left( t\right) \right) }\varphi \left( \gamma \right) d\gamma
\right) +\Psi \left( \int_{0}^{\omega _{0}^{L}\left( X\right) }\varphi
\left( \gamma \right) d\gamma \right)
\end{eqnarray*}%
using the fact that $\Psi $ is concave, we get%
\begin{eqnarray*}
&&\Psi \left( \int_{0}^{\lim \sup\limits_{t\rightarrow \infty }Diam\left(
X\left( t\right) \right) }\varphi \left( \gamma \right) d\gamma \right)
+\Psi \left( \int_{0}^{\omega _{0}^{L}\left( X\right) }\varphi \left( \gamma
\right) d\gamma \right) \\
&=&\Psi \left( \frac{1}{2}\int_{0}^{\lim \sup\limits_{t\rightarrow \infty
}Diam\left( X\left( t\right) \right) }2\varphi \left( \gamma \right) d\gamma
\right) +\Psi \left( \frac{1}{2}\int_{0}^{\omega _{0}^{L}\left( X\right)
}2\varphi \left( \gamma \right) d\gamma \right) \\
&\leqslant &\Psi \left( \frac{1}{2}\left( \int_{0}^{\lim
\sup\limits_{t\rightarrow \infty }Diam\left( X\left( t\right) \right)
}2\varphi \left( \gamma \right) d\gamma +\int_{0}^{\omega _{0}^{L}\left(
X\right) }2\varphi \left( \gamma \right) d\gamma \right) \right) \\
&=&\Psi \left( \frac{1}{2}\int_{0}^{\mu \left( TX\right) }2\varphi \left(
\gamma \right) d\gamma \right)
\end{eqnarray*}%
Finally, we obtain%
\begin{equation*}
\Phi \left( \int_{0}^{\mu \left( TX\right) }\varphi \left( \gamma \right)
d\gamma \right) \leqslant \Psi \left( \int_{0}^{\mu \left( X\right) }\varphi
\left( \gamma \right) d\gamma \right) ,
\end{equation*}%
and $T$ has a fixed point.
\end{proof}

\begin{example}
let the following integral equation%
\begin{equation}
x\left( t\right) =\sin t+\ln \left( 1+\int_{0}^{t}\frac{1}{t^{2}+1}%
e^{-s^{2}}\cos x\left( t\right) ds\right) +\ln \left( 1+x\left( t\right)
\right) .  \label{exple}
\end{equation}%
Considering that by putting%
\begin{equation*}
f\left( t,x,y\right) =\sin t+\ln \left( 1+x\right) +\ln \left( 1+y\right)
\end{equation*}%
and%
\begin{equation*}
g\left( t,s,x\right) =\frac{1}{t^{2}+1}e^{-s^{2}}\cos x,
\end{equation*}%
we obtain an equation of the form $\left( \text{\ref{eq}}\right) $ and it
satisfies assumptions $\left( i-v\right) $. Indeed, it is to see that
assumption $\left( i\right) $ is satisfied and by taking $\Phi \left(
x\right) =kx$ for $k\in \left[ 0,1\right[ $, $\Psi \left( x\right) =\ln
\left( 1+x\right) $ and $\varphi \left( t\right) =1$, we get%
\begin{eqnarray*}
\Phi \left( \int_{0}^{\left\vert f\left( t,x,y_{1}\right) -f\left(
t,x,y_{2}\right) \right\vert }\varphi \left( \gamma \right) d\gamma \right)
&=&k\left\vert f\left( t,x,y_{1}\right) -kf\left( t,x,y_{2}\right)
\right\vert \\
&=&k\ln \left( \frac{1-\left\vert y_{1}\right\vert }{1-\left\vert
y_{2}\right\vert }\right) \\
&=&k\ln \left( 1+\frac{1+\left( \left\vert y_{2}\right\vert -\left\vert
y_{1}\right\vert \right) }{1-\left\vert y_{2}\right\vert }\right) \\
&\leqslant &k\ln \left( 1+\left( \left\vert y_{2}\right\vert -\left\vert
y_{1}\right\vert \right) \right) \\
&\leqslant &\ln \left( 1+\left\vert y_{2}-y_{1}\right\vert \right) \\
&=&\Psi \left( \left\vert y_{2}-y_{1}\right\vert \right) .
\end{eqnarray*}%
Hence,%
\begin{equation*}
\Phi \left( \int_{0}^{\left\vert f\left( t,x,y_{1}\right) -f\left(
t,x,y_{2}\right) \right\vert }\varphi \left( \gamma \right) d\gamma \right)
\leqslant \Psi \left( \int_{0}^{\left\vert y_{1}-y_{2}\right\vert }\varphi
\left( \gamma \right) d\gamma \right) .
\end{equation*}%
The same way we prove that,%
\begin{equation*}
\Phi \left( \int_{0}^{\left\vert f\left( t,x_{1},y\right) -f\left(
t,x_{2},y\right) \right\vert }\varphi \left( \gamma \right) d\gamma \right)
\leqslant \Psi \left( \int_{0}^{\left\vert x_{1}-x_{2}\right\vert }\varphi
\left( \gamma \right) d\gamma \right) .
\end{equation*}%
Then, assumption $\left( ii\right) $ holds.

In further,%
\begin{eqnarray*}
\left\vert g\left( t,s,x\right) \right\vert &=&\left\vert \frac{1}{t^{2}+1}%
e^{-s^{2}}\sin x\right\vert \\
&\leqslant &\left\vert \frac{1}{t^{2}+1}e^{-s^{2}}\right\vert ,
\end{eqnarray*}%
then by taking $a\left( t\right) =\frac{1}{t^{2}+1}$ and $b\left( s\right)
=e^{-s^{2}},$ it is easy to see that $g$, $a$ and $b$ satifies the
conditions in $\left( iii\right) .$

Finally,%
\begin{equation*}
\int_{0}^{r}\varphi \left( \gamma \right) d\gamma \leqslant
\int_{0}^{r}\varphi \left( \gamma \right) d\gamma +M_{0}+M_{1},
\end{equation*}%
where $M_{0}=\sup \left\{ \int_{0}^{\left\vert a\left( t\right) \right\vert
\int_{0}^{t}\left\vert b\left( s\right) \right\vert ds}\varphi \left( \gamma
\right) d\gamma \right\} $ and $M_{1}=\Phi \left( \int_{0}^{\sup \left\vert
f\left( t,0,0\right) \right\vert }\varphi \left( \gamma \right) d\gamma
\right) $, holds for every positive $r.$
\end{example}

\end{document}